\documentclass[graybox]{svmult}

\usepackage{mathptmx}       % selects Times Roman as basic font
\usepackage{helvet}         % selects Helvetica as sans-serif font
\usepackage{courier}        % selects Courier as typewriter font
\usepackage{type1cm}        % activate if the above 3 fonts are
\usepackage{makeidx}         % allows index generation
\usepackage{graphicx}        % standard LaTeX graphics tool
\usepackage{multicol}        % used for the two-column index
\usepackage[bottom]{footmisc}% places footnotes at page bottom

\makeindex             % used for the subject index

\usepackage{amsmath,amsfonts,amssymb,times}
\usepackage{latexsym}

\DeclareMathAlphabet{\curly}{U}{rsfs}{m}{n}  %% curly font

\newtheorem{lem}{Lemma}[section]
\newtheorem{thm}{Theorem}

\numberwithin{equation}{section}

\newcommand{\Li}{{\rm Li}}

\makeatletter
\renewcommand{\pmod}[1]{\allowbreak\mkern7mu({\operator@font mod}\,\,#1)}
\makeatother

\newcommand{\bal}{\[\begin{aligned}}
\newcommand{\eal}{\end{aligned}\]}

\newcommand{\be}{\begin{equation}}
\newcommand{\ee}{\end{equation}}

\newcommand{\ssum}[1]{\sum_{\substack{#1}}}  %%% stacked sum
\newcommand{\sprod}[1]{\prod_{\substack{#1}}}  %%% stacked product

\newcommand{\eps}{\ensuremath{\varepsilon}}

\renewcommand{\le}{\leqslant}
\renewcommand{\leq}{\leqslant}
\renewcommand{\ge}{\geqslant}
\renewcommand{\geq}{\geqslant}

\renewcommand{\(}{\left(}
\renewcommand{\)}{\right)}

\newcommand{\pfrac}[2]{\left(\frac{#1}{#2}\right)}  %%% frac with paren

\newcommand{\PP}{\mathcal{P}}

\begin{document}

\author{Kevin Ford, D. R. Heath-Brown and Sergei Konyagin}

\title*{Large gaps between consecutive prime numbers containing perfect
powers}

\institute{Ford \at Department of Mathematics, 1409 West Green Street, University
of Illinois at Urbana-Champaign, Urbana, IL 61801, USA, \email{ford@math.uiuc.edu}
\and Heath-Brown \at Mathematical Institute,
Radcliffe Observatory Quarter, Woodstock Road, Oxford, OX2 6GG, UK
\email{rhb@maths.ox.ac.uk}
\and Konyagin \at Steklov Mathematical Institute,
8 Gubkin Street, Moscow, 119991, Russia \email{konyagin@mi.ras.ru}}

\maketitle

\abstract{For any positive integer $k$,
we show that infinitely often, perfect $k$-th
powers appear inside very long gaps between consecutive prime numbers,
that is, gaps of size
$$
c_k \frac{\log p \log_2 p \log_4 p}{(\log_3 p)^2},
$$
where $p$ is the smaller of the two primes.}

%%%%%%%%%%%%%%%%%%%%%%%%%

\section{Introduction}
\label{sec:1}
In 1938, Rankin \cite{Ra38} proved that the maximal gap, $G(x)$, between primes
$\leq x$, satisfies\footnote{As usual in the subject, $\log_2 x = \log
  \log x$, $\log_3 x = \log \log \log x$, and so on.}
\be\label{Rankin}
G(x) \geq (c+o(1)) \frac{\log x \log_2 x \log_4x}{(\log_3 x)^2},
\ee
with $c=\frac13$.  The following six decades witnessed several improvements
of the constant $c$; we highlight out only a few of these.  First,
Rankin's own improvement \cite{Ra63} $c = e^{\gamma}$ in 1963
represented the limit of what could be achieved by inserting into
Rankin's original 1938 argument best possible bounds on counts of
``smooth'' numbers.  This record stood for a long time until
 Maier and Pomerance \cite{MP} introduced new ideas to improve the
 constant to $c = 1.31256 e^{\gamma}$ in 1989; these were refined by Pintz
 \cite{Pi}, who obtained $c = 2e^{\gamma}$ in 1997.
Very recently, the first and third
authors together with B. Green and T. Tao \cite{FGKT}
have shown that $c$ can be taken arbitrarily large. Independently,
this was also proven by J. Maynard \cite{M}.

% who actually established that
%$$
%G(x) \gg \frac{\log x \log_2 x}{\log_3 x}.
%$$

Rankin's lower bound \eqref{Rankin} is probably very far from the
truth.
Based on a probabilistic model of primes, Cram\'er \cite{Cra} conjectured
that
\[
\limsup_{X\to\infty} \frac{G(X)}{\log^2 X} = 1,
\]
and Granville \cite{Gra}, using a refinement of Cram\'er's model,
has conjectured that the $\limsup$ above is in fact at least $2e^{-\gamma}=1.229\ldots$.
Cram\'er's model also predicts that the normalized prime gaps
$\frac{p_{n+1}-p_n}{\log p_n}$ should have exponential distribution, that is,
$p_{n+1}-p_n \ge C\log p_n$ for about $e^{-C}\pi(X)$ primes $\le X$.

Our aim in this paper is to study whether or not long prime gaps, say
of the size of the right hand size of the inequality in
\eqref{Rankin}, occur when we impose that an integer of a specified type
lies inside the interval.  To be precise,
we say that a number $m$ is ``prime avoiding with constant $c$''
if $m+u$ is composite for all integers $u$ satisfying
\[
 |u| \le c \frac{\log m \log_2 m \log_4 m}{(\log_3 m)^2}.
\]
Here we will be concerned with prime avoiding perfect powers.

\begin{thm}\label{mainthm}
For any positive integer $k$,
there are a constant $c=c(k)>0$ and infinitely
many perfect $k$-th powers which are prime-avoiding with constant $c$.
\end{thm}

%%%%%%%%%%%%%%%%%%%%%%
%
\section{Sieve estimates}
\label{sec:2}
%
%%%%%%%%%%%%%%%%%%%%%%
Throughout, constants implied by the Landau $O$-symbol and Vinogradov
$\ll$-symbol are absolute unless otherwise indicated, e.g. by a subscript
such as $\ll_u$.  The symbols $p$ and $q$ will always denote prime numbers.
Denote by $P^+(n)$ the largest prime factor of a positive integer
$n$, and by $P^-(n)$ the smallest  prime factor of $n$.

We need several standard lemmas from sieve theory, the distribution of
 ``smooth'' numbers, and the distribution of primes in arithmetic
progressions.

\begin{lem}\label{psi}
For large $x$ and $z \le x^{\log_3 x/(10\log_2 x)}$, we have
\[
\# \{ n\le x : P^+(n) \le z \} \ll \frac{x}{\log^5 x}.
\]
\end{lem}

\begin{proof}
 This follows from standard counts of smooth numbers.  Lemma 1 of
Rankin \cite{Ra38} also suffices.
\qed
\end{proof}

\begin{lem}\label{sieve1modr}
Let $\mathcal{R}$ denote any set of primes and let $a\in \{-1,1\}$.  Then
\[
\# \{ p\le x : p\not \equiv a\pmod{r} \; (\forall r\in \mathcal{R}) \}
\ll \frac{x}{\log x} \sprod{p\in \mathcal{R}\\ p\le x} \(1-\frac{1}{p}\).
\]
\end{lem}

\begin{proof}
Standard sieve methods \cite{HR}.
\qed
\end{proof}

Finally, we require a bound of ``large sieve''
type for averages of quadratic character sums.

\begin{lem}\label{HBlemma}
For any set $\mathcal{P}$ of primes in $[2,x]$, and for any $\eps>0$,
\[
\ssum{m\le x \\ m\text{ odd}} \mu^2(m) \Big| \sum_{p\in \mathcal{P}} \pfrac{p}{m}
\Big|^2 \ll_\eps x^{2+\eps}.
\]
\end{lem}

\begin{proof}
This follows immediately from Theorem 1 of \cite{H-B}.
\qed
\end{proof}

%%%%%%%%%%%%%%%%%%%%%%%%%%%%%%%%%%%%
%
\section{$k$-th power residues and prime ideals}
\label{sec:PIT}
%
%%%%%%%%%%%%%%%%%%%%%%%%%%%%%%%%%%%%

One of our principle tools is the following estimate for an average of
counts of solutions of a certain $k$-th power congruence.

\begin{lem}\label{rhop}
 Let $k$ be a positive integer.  For any non-zero integer $u$ and any
 prime $p$ write
\[\rho_{u,k}(p)=\rho_u(p)=\#\{n\pmod{p}:\,n^k+u\equiv 0\pmod{p}\}.\]
%SK Then for any fixed $\eps>0$ we have
Then for any fixed $\eps>0$ and $x\ge2$ we have
%SK If $x=1$, then $\log x=0$ and the innequality fails.
\[\prod_{x<p\le y}\left(1-\frac{\rho_u(p)}{p}\right)\ll_{k,\eps}
|u|^{\eps}\frac{\log x}{\log y}.\]
\end{lem}

The proof of this is based on the Prime Ideal Theorem, and we begin by
giving a formal statement of an appropriate form of the latter.

\begin{lem}
There is an effectively computable absolute constant $c>0$ with the
following property.
Let $K$ be an algebraic number field of degree $n_K$, and write $d_K$
for the absolute value of the discriminant of $K$. Let $\beta_0$ be
the largest simple real zero of $\zeta_K(s)$ in the interval $[\tfrac12,1]$
if any such exists. Then
\[|\pi_K(x)-\Li(x)|\le\Li(x^{\beta_0})+
cx\exp\{-cn_K^{-1/2}(\log x)^{1/2}\}\]
for $x\ge\exp\{10n_K(\log d_K)^2\}$,
where as usual, $\pi_K(x)$ denotes the number of prime ideals of $K$
with norm at most $x$, and $\Li(x)=\int_2^x dt/\log t$.
%SK,KF
We omit the first summand on the right hand side if $\beta_0$ does not exist.
\end{lem}

This follows from Theorem 1.3 of Lagarias and Odlyzko \cite{LO}, on
choosing $L=K$ in their notation.  The reader should note that the
counting function $\pi_C(x,L/K)$ of \cite{LO} excludes ramified
primes, but the number of these is $O(n_K\log d_K)$, which is
majorized by $x\exp\{-cn_K^{-1/2}(\log x)^{1/2}\}$.

In order to handle the term involving the possible simple real zero $\beta_0$
we use the following result of Heilbronn \cite[Theorem 1]{Heil}.
\begin{lem}
A simple real zero of $\zeta_K(s)$ must be a zero of $\zeta_k(s)$ for
some quadratic subfield $k$ of $K$.
\end{lem}

It follows that $\beta_0$ is a zero for some quadratic Dirichlet
L-function $L(s,\chi)$, with a character $\chi$ of conductor dividing
$d_K$. Thus Siegel's Theorem shows that $1-\beta_0\ge c(\eps)
d_K^{-\eps}$, for any fixed $\eps>0$, with an ineffective constant
$c(\eps)>0$.  We then deduce that
\[\Li(x^{\beta_0})\ll x^{\beta_0}\le x\exp\{-c(\eps)d_K^{-\eps}\log x\}
\le x\exp\{-c(\eps)n_K^{-1/2}(\log x)^{1/2}\}\]
if $n_K\log x\ge d_K^{2\eps}$.  We therefore obtain the following
version of the Prime Ideal Theorem.

\begin{lem}\label{lem_n1}
For any $\eta>0$ there is an ineffective constant $C(\eta)>0$ with the
following property.
Let $K$ be an algebraic number field of degree $n_K$, and write $d_K$
for the absolute value of the discriminant of $K$. Then
\[\pi_K(x)=\Li(x)+O(x\exp\{-C(\eta)n_K^{-1/2}(\log x)^{1/2}\})\]
for
\[x\ge\exp\left\{\max\big(10n_K(\log d_K)^2\,,\,n_K^{-1}d_K^\eta\big)\right\}.\]
\end{lem}

\begin{proof}[Proof of Lemma \ref{rhop}.]  Since it may happen that
$-u$ is a perfect power we begin by taking $a$ to be the largest
divisor of $k$ for which $-u$ is a perfect $a$-th power. Then if $-u=v^a$
and $k=ab$ we see firstly that the polynomial $X^b-v$ is irreducible
over the rationals, and secondly that $n^b\equiv v\pmod{p}$
implies $n^k+u\equiv 0\pmod{p}$, whence
\begin{equation}\label{E1}
\rho_{u,k}(p)\ge\rho_{-v,b}(p).
\end{equation}We will apply Lemma \ref{lem_n1} to the field
$K=\mathbb{Q}(\theta)$, where $\theta$ is a root of $X^b-v$.  Thus $K$
has degree $b\le k$.  Moreover its discriminant will be a divisor of
\[D:={\rm  Disc}(1,\theta,\theta^2,\ldots,\theta^{b-1})=(-1)^{b-1}b^bv^{b-1}.\]

We now set
\[x_0=C(k,\eta)\exp(|D|^{\eta}).\]
If we choose the constant $C(k,\eta)$ sufficiently large then whenever
$x\ge x_0$ we will have
\[x\exp\{-C(\eta)b^{-1/2}(\log x)^{1/2}\} \le x(\log x)^{-2}\]
and
\[x\ge\exp\left\{\max\big(10b(\log|D|)^2\,,\,b^{-1}|D|^\eta\big)\right\}.\]
It therefore follows from Lemma \ref{lem_n1} that
\[\pi_K(x)=\Li(x)+O_{k,\eta}\(x(\log x)^{-2}\)\]
for $x\ge x_0$.

We now write $\nu_K(p)$ for
the number of first degree prime ideals of $K$ lying above $p$.  Then
\[\pi_K(x)=\sum_{p\le x}\nu_K(p)+O_k(\sum_{p^e\le x,\, e\ge 2}1).\]
Moreover, by Dedekind's Theorem we will have $\rho_{-v,b}(p)=\nu_K(p)$
whenever $p\nmid D$. In the remaining case in which $p\mid D$ we have
$\rho_{-v,b}(p)\le b\le k$ and $\nu_K(p)\le b\le k$.  It therefore follows that
\[\pi_K(x)=\sum_{p\le x}\rho_{-v,b}(p)+O_k(x^{1/2})+O_k(\log|D|),\]
so that
\[\sum_{p\le x}\rho_{-v,b}(p)=\Li(x)+O_{k,\eta}(x(\log x)^{-2})\]
when $x\ge x_0$.

We now observe that
\begin{eqnarray*}
\prod_{x<p\le y}\left(1-\frac{\rho_u(p)}{p}\right)&\le&
\exp\left\{-\sum_{x<p\le y}\frac{\rho_u(p)}{p}\right\}\\
&\le& \exp\left\{-\sum_{x<p\le y}\frac{\rho_{-v,b}(p)}{p}\right\}
\end{eqnarray*}
by (\ref{E1}).
Assuming that $y\ge x_0$ we may then use summation by parts to calculate that
\begin{eqnarray*}
\sum_{x<p\le y}\frac{\rho_{-v,b}(p)}{p}&\ge&
\sum_{\max(x,x_0)<p\le y}\frac{\rho_{-v,b}(p)}{p}\\
&=&\log\log y-\log\log\big(\max(x,x_0)\big)+O_{k,\eta}(1)\\
&\ge&\log\log y-\log\log x-\log\log x_0+O_{k,\eta}(1)\\
&=&\log\log y-\log\log x-\eta\log|D|+O_{k,\eta}(1)\\
&\ge&\log\log y-\log\log x-\eta k\log|u|+O_{k,\eta}(1).
\end{eqnarray*}
We therefore have
\[\prod_{x<p\le y}\left(1-\frac{\rho_u(p)}{p}\right)\ll_{k,\eta}
|u|^{k\eta}\frac{\log x}{\log y}\]
when $y\ge x_0$.  Of course this estimate is trivial when $y\le x_0$
since one then has $\log y\ll_{k,\eta}
|D|^{\eta}\ll_{k,\eta}|u|^{k\eta}$. The lemma then follows.
\end{proof}

%%%%%%%%%%%%%%%%%%%%%%%%%%%%%%%%%%%%
%
\section{Main argument}
\label{sec:3}
%
%%%%%%%%%%%%%%%%%%%%%%%%%%%%%%%%%%%%
Fix a positive integer $k$.
Let $x$ be a large number, sufficiently large depending on $k$, let
$c_1$ and $c_2$ be two positive constants depending on $k$ to be chosen later,
and put
\[
N = \prod_{p\le x} p, \qquad z=x^{c_1 \log_3 x/\log_2 x}, \qquad
 y = \frac{c_2 x \log x \log_3 x}{(\log_2 x)^2}.
\]
%SK Duplicated!
%SK We will show that there is a number $m\le 2N$ such that
%SK $m^k+u$ is composite for $|u|\le y$.  Theorem \ref{mainthm} will
%SK follow upon observing that $m=e^{kx+o(x)}$ as $x\to\infty$ and consequently
%SK that
%SK \[
%SK y \gg_k \frac{\log m \log_2 m \log_4 m}{(\log_3 m)^2}.
%SK \]

In the rest of the paper we will prove the following lemma.
\begin{lem}\label{main} There is a number $m\le 2N$ such that
$m^k+u$ is composite for $|u|\le y$.
\end{lem}

Theorem \ref{mainthm} will
% KF : was $m = e^{kx+o(x)}$
follow upon observing that $m^k \le e^{kx+o(x)}$ as $x\to\infty$ and consequently
that
%KF
$$
y \gg_k \frac{\log (m^k) \log_2 (m^k) \log_4 (m^k)}{(\log_3 (m^k))^2}.
$$

We will select $m$ be choosing residue classes for $m$ modulo $p$ for
primes $p\le x$.  Let
$$
\PP_1 = \{ p: p\le \log x \text{ or } z < p \le x/4 \}, \qquad
\PP_2 = \{ p: \log x < p \le z \}.
$$
We first choose
\be\label{m1}\begin{split}
 m&\equiv 0\pmod{p} \quad (p\in\PP_1), \\
 m&\equiv 1\pmod{p} \quad (p\in\PP_2).
\end{split}
\ee
Observe that $p|(m^k+u)$ if $p|u$ for some $p\in \PP_1$.
Because $y< (x/4)\log x$, any remaining value of $u$ is thus either
composed only of 
primes in $\PP_2$ (in particular, $u$ is $z$-smooth), including
$|u|=1$, or $|u|$ is a prime larger than $x/4$. 
 For any $u$ in the latter category such that $p|(u+1)$ for some $p\in
\PP_2$, $p|(m^k+u)$. 
 Let $U$ denote the set of exceptional values of $u$, that is, the set of $u\in [-y,y]$
 not divisible by any prime in $\PP_1$, and such that if $|u|$ is prime then $p\nmid (u+1)$ for all $p\in\PP_2$.
 By Lemmas \ref{psi} and \ref{sieve1modr}, if $c_1$ is sufficiently small, then
 \[
|U|  \ll \frac{y}{\log^5 x} + \frac{y}{\log x} \prod_{p\in \PP_2} \(1-\frac{1}{p}\) \ll
  \frac{y\log_2 x}{\log x\log z} = \frac{c_2}{c_1} \frac{x}{\log x}.
 \]
Choosing $c_2$ appropriately, we can ensure that $|U| \le \delta x/\log x$, where $\delta>0$ depends on
$k$ ($\delta$ will be chosen later).

The remaining steps depend on whether $k$ is odd or even.  If $k$ is odd, the construction is very easy.
For each $u\in U$, associate with $u$ a different prime $p_u\in (x/4,x]$ such that $(p_u-1,k)=1$ (e.g., one can take
$p_u\equiv 2\pmod{k}$ if $k\ge 3$).  Then every residue modulo $p_u$ is a $k$-th power residue, and we take $m$
in the residue class modulo $p_u$ such that
\be\label{m2kodd}
 m^k \equiv -u \pmod{p_u} \quad (u\in U).
\ee
By the prime number theorem for arithmetic progressions, the number of
available primes is at least
\[
x/(2\phi(k)\log x)\ge |U|
\]
if $\delta$ is small enough.
%SK With this construction, $p_u|(m^k+u)$ for every $u\in U$.  Therefore, $m^k+u$ is composite for $|u|\le y$.
With this construction, $p_u|(m^k+u)$ for every $u\in U$. 
% KF
 Therefore, $m^k+u$ is divisible by a prime $\le x$ for every $|u|\le y$.
%SK Furthermore, \eqref{m1} and \eqref{m2kodd} together imply define $m$ modulo a number $N'$, where $N'|N$.
Furthermore, \eqref{m1} and \eqref{m2kodd} together imply that $m$ is defined
modulo a number $N'$, where $N'|N$.
% KF new
Therefore, there is an admissible value of $m$ satisfying 
$N < m\le 2N$.  The prime number theorem implies that $N=e^{x+o(x)}$,
thus $m^k-y > x$.  Consequently, $m^k+u$ is
composite for $|u| \le y$.

Now suppose that $k$ is even.
There do not exist primes for which every residue modulo $p$ is a $k$-th power residue.
However, we maximize the density of $k$-th power residues by choosing primes $p$ such that
$(p-1,k)=2$, e.g. taking $p\equiv 3\pmod{2k}$.  For such primes $p$, every quadratic residue is a $k$-th power residue.
Let
\[
%SK \PP_3 = \{ x/4 < p \le p/2 : p\equiv 3\pmod{2k} \}.
 \PP_3 = \{ x/4 < p \le x/2 : p\equiv 3\pmod{2k} \}.
\]
%SK By the prime number theorem for arithmetic progressions, $|\PP_3|\ge x/(5\phi(k)\log x)$.
By the prime number theorem for arithmetic progressions, $|\PP_3|\ge x/(5\phi(2k)\log x)$.
We aim to associate numbers $u\in U$ with distinct primes $p_u\in \PP_3$ such that $\pfrac{-u}{p_u}=1$.
This ensures that the congruence $m^k+u\equiv 0\pmod{p_u}$ has a solution.  We, however, may
not be able to find such $p$ for every $u\in U$, but can find appropriate primes for most $u$.
Let
\[
 U' = \left\{u\in U: \pfrac{-u}{p}=1\text{ for at most } \frac{\delta x}{\log x} \text{ primes } p\in \PP_3 \right\}.
\]
The numbers $u\in U\backslash U'$ may be paired with different primes $p_u\in \PP_3$ such that
$\pfrac{-u}{p_u}=1$.  We then may take $m$ such that
\be\label{m2}
 m^k \equiv -u \pmod{p_u} \quad (u\in U\backslash U').
\ee
%SK It remains to show that $|U'|$ is small.   Write
Next we will show that $|U'|$ is small.   Write
\[
 S = \sum_{u\in U} \Big| \sum_{p\in \PP_3} \pfrac{-u}{p} \Big|^2.
\]
%SK Each $u$ may be written uniquely in the form $u=su_1^2 u_2$, where $s=\pm 1$, $(u_1,u_2)=1$ and $u_2$ is squarefree.
Each $u$ may be written uniquely in the form $u=su_1^2 u_2$, where $s=\pm 1$,
$u_2>0$ and $u_2$ is squarefree.
By quadratic reciprocity,
\[
 \pfrac{-u}{p} = (-s) \pfrac{u_2}{p} = (-s)(-1)^{\frac{u_2-1}2} \pfrac{p}{u_2},
\]
since $p\equiv 3\pmod{4}$.
Given $u_2$, there are at most $\sqrt{y/u_2} \le \sqrt{y}$ choices for $u_1$.  Hence, using Lemma \ref{HBlemma},
\begin{align*}
 S &= \sum_{u\in U} \Big| \sum_{p\in \PP_3} \pfrac{p}{u_2} \Big|^2 \\
&\le \sum_{u_2\le y} 2y^{1/2} \Big| \sum_{p\in \PP_3} \pfrac{p}{u_2} \Big|^2 \\
&\ll_\eps x^{5/2+\eps}.
 \end{align*}
%SK Now let $\delta = \frac{1}{15\phi(k)}$, so that $\delta x/\log x \le \frac13 |\PP_3|$. If $u\in U'$,
Now let $\delta = \frac{1}{15\phi(2k)}$, so that $\delta x/\log x \le \frac13 |\PP_3|$. If $u\in U'$,
then clearly
\[
\Big| \sum_{p\in \PP_3} \pfrac{-u}{p} \Big| \ge \frac13 |\PP_3 |  \ge \delta \frac{x}{\log x}.
\]
It follows that $|S| \gg |U'| (x/\log x)^2$, and consequently that
\be\label{Uprime}
|U'| \ll_\eps x^{1/2+2\eps}.
\ee

Let $A\mod M$ denote the set of numbers $m$ satisfying the congruence conditions \eqref{m1} and
\eqref{m2}, where $0\le A< M$.  Thus, if $m\equiv A\pmod{M}$ and $u\not\in U'$, then $m^k+u$ is divisible by a prime $\le x/2$.
Let
\[
 K = \prod_{x/2 <p\le x} p.
\]
We'll take $m=Mj+A$, where $1\le j\le K$, and aim to show that there exists a value of $j$ so that
$(mj+A)^k+u$ is composite for every $u\in U'$.
By sieve methods (see \cite{HR}),
\begin{align*}
 \sum_{j=1}^K \#\{ u\in U' : (Mj+A)^k+u\text{ prime} \} &= \sum_{u\in U'} \# \{ 1\le j\le K : (Mj+A)^k+u \text{ prime} \} \\
%SK &\ll \sum_{u\in U'} K \prod_{y < q \le \sqrt{K}} \(1 - \frac{\rho_u(p)}{q} \).
 &\ll \sum_{u\in U'} K \prod_{y < q \le \sqrt{K}} \(1 - \frac{\rho_u(q)}{q} \).
\end{align*}
By Lemma \ref{rhop}, the above product is
$$\ll_{k,\eps} u^{\eps/2} \frac{\log y}{\log K} \ll_{k,\eps} u^{\eps/2} \frac{\log x}{x}. $$
Combined with our estimate
\eqref{Uprime} for the size of $|U'|$, we find that
\[
 \sum_{1\le j\le K} \#\{ u\in U' : (Mj+A)^k+u\text{ prime} \}
\ll_{k,\eps} \frac{K}{x^{1/2-4\eps}}.
\]
It follows that the left hand side above is zero for some $j$.  That is, $(Mj+A)^k+u$ is composite for every
$u\in U'$.  Therefore, $(Mj+A)^k+u$ is composite for every $u$
satisfying $|u| \le y$.  Finally, we note that $Mj+A \le 2N$, and the proof of
Lemma~\ref{main} is complete.

\bigskip

\textbf{Remark 1.} For odd $k$ the constant $c(k)$ in Theorem~\ref{mainthm}
is effective. For even $k$ it is ineffective due to the use of Siegel's 
theorem in the proof of Lemma~\ref{rhop}.

\begin{acknowledgement}
Research of the third author was partially performed while he was
visiting the University of Illinois at
Urbana--Champaign.  Research of the first author was partially performed
while visiting the University of Oxford.
 Research of the first and third authors was
also carried out in part at the University of Chicago, 
and they are thankful to Prof. Wilhelm Schlag
for hosting these visits.  

The first author was supported by NSF grant DMS-1201442.
The research of the second author was supported by EPSRC grant EP/K021132X/1.
The research of the the third author was partially supported by
Russian Foundation for Basic Research, Grant 14-01-00332, and
by Program Supporting Leading Scientific Schools,
Grant Nsh-3082.2014.1.
\end{acknowledgement}

%%%%%%%%%%%%%%%%%%%%%%%%%


\begin{thebibliography}{999}


\bibitem{Cra} H. Cram\'er, \emph{On the order of magnitude of the difference
between consecutive prime numbers,} Acta Arith. \textbf{2} (1936),
396--403.

\bibitem{Da} H. Davenport, {\it Multiplicative number theory}, 3rd ed.,
Graduate Texts in Mathematics vol. 74, Springer-Verlag, New York, 2000.

\bibitem{Gra} A. Granville, \emph{Harald Cram\'er and the distribution of
prime numbers}, Scandanavian Actuarial J. \textbf{1}  (1995), 12--28.

\bibitem{FGKT} K. Ford, B. Green, S. Konyagin and T. Tao,
\emph{Large gaps between consecutive prime numbers}, 
\texttt{arXiv:1408.4505}.

\bibitem{HR} H. Halberstam and H.-E. Richert, \textit{Sieve Methods},
Academic Press, London, 1974.

\bibitem{H-B} D. R. Heath-Brown, \textit{A mean value estimate for real
character sums}, Acta Arith. \textbf{72} (1995), 235--275.

\bibitem{Heil}
H. Heilbronn, \textit{On real zeros of Dedekind $\zeta$-functions, }
Canad. J. Math. \textbf{25} (1973), 870--873.

\bibitem{LO} 
J.C. Lagarias and A.M. Odlyzko, \textit{Effective versions of the 
Chebotarev density theorem,} Algebraic number fields, 
(Proc. Sympos., Univ. Durham, Durham, 1975), 409--464. 
(Academic Press, London, 1977).

\bibitem{MP} H. Maier and C. Pomerance, {\it Unusually large gaps
between consecutive primes}. Trans. Amer. Math. Soc. {\bf 322} (1990), 
no. 1, 201--237.

\bibitem{M}
%SK J. Maynard, \emph{Small gaps between primes}, preprint.
J. Maynard, \emph{Large gaps between primes}, \texttt{arXiv:1408.5110}.

\bibitem{Pi} J. Pintz, {\it Very large gaps between consecutive primes.}
J. Number Theory {\bf 63} (1997), no. 2, 286--301.

\bibitem{Ra38} R. Rankin, \textit{The difference between consecutive prime
numbers}, J. London Math. Soc. \textbf{13} (1938), 242--247.

\bibitem{Ra63} R. A. Rankin, \emph{The difference between consecutive
prime numbers. V,} Proc. Edinburgh Math. Soc. (2) 13 (1962/63), 331--332.


\end{thebibliography}
\end{document}